\documentclass[12pt]{amsart}

\usepackage{mathrsfs,amssymb}

\newtheorem{theorem}{Theorem}[section]

\newtheorem{prop}[theorem]{Proposition}

\theoremstyle{definition}

\theoremstyle{remark}

\numberwithin{equation}{section}

\let \O=\Omega

\begin{document}
\title[A note on weighted bounds]
{A note on weighted bounds for rough singular integrals}

\author{Andrei K. Lerner}
\address{Department of Mathematics,
Bar-Ilan University, 5290002 Ramat Gan, Israel}
\email{lernera@math.biu.ac.il}

\thanks{The author was supported by ISF grant No. 447/16 and ERC Starting Grant No. 713927.}

\begin{abstract}
We show that the $L^2(w)$ operator norm of the composition $M\!\circ T_{\O}$, where $M$ is the maximal operator and $T_{\O}$ is a rough homogeneous singular integral with
angular part $\O\in L^{\infty}(S^{n-1})$, depends quadratically on $[w]_{A_2}$, and this dependence is sharp.
\end{abstract}


\maketitle

\section{Introduction}
Consider a class of rough homogeneous singular integrals defined by
$$T_{\Omega}f(x)=\text{p.v.}\int_{{\mathbb R}^n}f(x-y)\frac{\O(y/|y|)}{|y|^n}dy,$$
with $\O\in L^{\infty}(S^{n-1})$ and having zero average over the sphere.

In \cite{HRT}, Hyt\"onen, Roncal and Tapiola proved that
\begin{equation}\label{quant}
\|T_{\O}\|_{L^2(w)\to L^2(w)}\le C_n\|\O\|_{L^{\infty}}[w]_{A_2}^2,
\end{equation}
where $[w]_{A_2}=\sup_Q\frac{\int_Qw\int_Qw^{-1}}{|Q|^2}$.
Different proofs of this result, via a sparse domination, were given by Conde-Alonso, Culiuc, Di Plinio and Ou \cite{CCPO},
and by the author \cite{L}. Recently (\ref{quant}) was extended to maximal singular integrals by
Di Plinio, Hyt\"onen and Li \cite{DHL}.

It was conjectured in \cite{HRT} that the quadratic dependence on $[w]_{A_2}$ in (\ref{quant}) can be improved to the linear one.
In this note we obtain a strengthening of (\ref{quant}), which, at some point, supports this conjecture.

\begin{theorem}\label{sharpbound} For every $w\in A_2$, we have
\begin{equation}\label{comp}
\|M\!\circ\!T_{\O}\|_{L^2(w)\to L^2(w)}\le C_n\|\O\|_{L^{\infty}}[w]_{A_2}^2,
\end{equation}
and this bound is optimal, in general.
\end{theorem}

Here $M$ denotes the standard Hardy-Littlewood maximal operator. Notice that $\|M\|_{L^2(w)\to L^2(w)}\lesssim [w]_{A_2}$,
and this bound is sharp~\cite{B}. Therefore, (\ref{comp}) cannot be obtained via a simple combination of the sharp linear bound for $M$ with (\ref{quant}). The proof
of (\ref{comp}) is based essentially on the technique introduced in \cite{L}.

\section{Preliminaries}
Recall that a family of cubes ${\mathcal S}$ is called sparse if there exists $0<\alpha<1$ such that for every
$Q\in {\mathcal S}$, one can find a measurable set $E_Q\subset Q$ with $|E_Q|\ge \alpha|Q|$, and the sets
$\{E_Q\}_{Q\in {\mathcal S}}$ are pairwise disjoint.

Given a sublinear operator $T$, define the maximal operator $M_{p,T}$ by
$$M_{p,T}f(x)=\sup_{Q\ni x}\left(\frac{1}{|Q|}\int_Q|T(f\chi_{{\mathbb R}^n\setminus 3Q})|^pdy\right)^{1/p}.$$
Denote $\langle f\rangle_{p,Q}=\left(\frac{1}{|Q|}\int_Q|f|^p\right)^{1/p}$.

\begin{prop}\label{pc1}
Assume that $T$ and $M_{p,T}$ are of weak type $(1,1)$ and, moreover, $\|M_{p,T}\|_{L^1\to L^{1,\infty}}\le Kp$ for all $p\ge 2$.
Then
\begin{equation}\label{qb}
\|T\|_{L^2(w)\to L^2(w)}\le C_n(\|T\|_{L^1\to L^{1,\infty}}+K)[w]_{A_2}^2.
\end{equation}
\end{prop}

\begin{proof} This is just a combination of several known facts. By \cite[Cor.~3.2]{L},
for every suitable $f,g$, there exists a sparse family ${\mathcal S}$ such that
$$|\langle Tf,g \rangle|\le C_n(\|T\|_{L^1\to L^{1,\infty}}+Kp')\sum_{Q\in {\mathcal S}}\langle f\rangle_{1,Q}\langle g\rangle_{p,Q}|Q|\quad(p>1).$$
But it was shown in \cite{CCPO} (see the proof of Corollary A1 there) that this sparse bound implies (\ref{qb}).
\end{proof}

In particular, $T_{\O}$ with $\O\in L^{\infty}$ satisfies the hypothesis of Proposition~\ref{pc1}, namely, it was proved in \cite{L} that
\begin{equation}\label{mp}
\|M_{p,T_{\O}}f\|_{L^{1,\infty}}\le C_n\|\O\|_{L^{\infty}}p\|f\|_{L^1}\quad (p\ge 1).
\end{equation}

\section{Proof of Theorem \ref{sharpbound}}
First, by a general extrapolation argument found in \cite{LPR}, the sharpness of (\ref{comp})
follows from $\|M\!\circ T_{\O}\|_{L^p\to L^p}\ge \frac{c}{(p-1)^2}$ as $p\to 1$. The latter relation holds for a subclass of
$T_{\O}$ with kernels satisfying the standard nondegeneracy assumptions. In particular, it can be easily checked for the Hilbert transform.

Turn to the proof of (\ref{comp}). By homogeneity, one can assume that $\|\O\|_{L^{\infty}}=1$. The proof is based on two pointwise estimates:
\begin{equation}\label{point1}
M(T_{\O}f)(x)\lesssim MMf(x)+M_{1,T_{\O}}f(x)
\end{equation}
and
\begin{equation}\label{point2}
M_{p,(M_{1,T_{\O}})}f(x)\lesssim Mf(x)+M_{p,T_{\O}}f(x)\quad(p\ge 2)
\end{equation}
(we use the usual notation $A\lesssim B$ if $A\le C_nB$).

Let us show first how to complete the proof using these estimates. By~(\ref{mp}), $M_{1,T_{\O}}$ is of weak type $(1,1)$. Applying (\ref{mp}) again
along with~(\ref{point2}) yields $\|M_{p,(M_{1,T_{\O}})}\|_{L^1\to L^{1,\infty}}\lesssim p$. Therefore, by Proposition~\ref{pc1},
$$\|M_{1,T_{\O}}\|_{L^2(w)\to L^2(w)}\lesssim [w]_{A_2}^2.$$
This estimate combined with (\ref{point1}) and Buckley's linear bound for $M$~\cite{B} implies (\ref{comp}).

It remains to prove (\ref{point1}) and (\ref{point2}). We start with (\ref{point1}). This estimate follows from the definition of $M_{1,T_{\O}}$ and
the standard fact that for every cube $Q$ containing the point $x$,
\begin{equation}\label{llog}
\frac{1}{|Q|}\int_{Q}|T_{\O}(f\chi_{3Q})|\lesssim MMf(x).
\end{equation}

For the sake of completeness we outline the proof of (\ref{llog}). Combining the weak type $(1,1)$ and the $L^2$ boundedness of $T_{\O}$ (see \cite{CZ,S}) with interpolation and
Yano's extrapolation \cite[p. 43]{G}, we obtain
$$\frac{1}{|Q|}\int_{Q}|T_{\O}(f\chi_{3Q})|\lesssim \|f\|_{L\log L,3Q}.$$
By Stein's $L\log L$ result \cite{St},
$$\|f\|_{L\log L,Q}\lesssim \frac{1}{|Q|}\int_QMf,$$
which, along with the previous estimate, implies (\ref{llog}).

Turn to the proof of (\ref{point2}).
Let $R$ be an arbitrary cube containing the point~$x$.
Let $y\in R$ and let $Q$ be an arbitrary cube containing $y$.

Assume that $\ell_Q\le \frac{1}{2}\ell_R$. Then $Q\subset 2R$ and $3Q\subset 3R$. Hence,
\begin{equation}\label{case1}
\frac{1}{|Q|}\int_Q|T_{\O}(f\chi_{{\mathbb R}^n\setminus (3R\cup 3Q)})|\le M_{2R}(T_{\O}(f\chi_{{\mathbb R}^n\setminus 3R}))(y).
\end{equation}

Suppose now that $\ell_R<2\ell_Q$. Then $R\subset 5Q$ and $3R\subset 9Q$. We obtain
$$|T_{\O}(f\chi_{15Q\setminus (3R\cup 3Q)})\chi_Q(z)|\lesssim \frac{1}{|Q|}\int_{15Q}|f|\lesssim Mf(x).$$
Also,
$$\frac{1}{|Q|}\int_Q|T_{\O}(f\chi_{{\mathbb R}^n\setminus 15Q})|\lesssim \frac{1}{|5Q|}\int_{5Q}|T_{\O}(f\chi_{{\mathbb R}^n\setminus 15Q})|\lesssim M_{1, T_{\O}}f(x),$$
and therefore,
$$
\frac{1}{|Q|}\int_Q|T_{\O}(f\chi_{{\mathbb R}^n\setminus (3R\cup 3Q)})|dz\lesssim Mf(x)+M_{1,T_{\O}}f(x).
$$

This estimate, combined with (\ref{case1}), implies
\begin{eqnarray*}
&&M_{1,T_{\O}}(f\chi_{{\mathbb R}^n\setminus 3R})(y)=\sup_{Q\ni y}\frac{1}{|Q|}\int_Q|T_{\O}(f\chi_{{\mathbb R}^n\setminus (3R\cup 3Q)})|\\
&&\lesssim M_{2R}(T_{\O}(f\chi_{{\mathbb R}^n\setminus 3R}))(y)+Mf(x)+M_{1,T_{\O}}f(x).
\end{eqnarray*}
Therefore, by the $L^p$-boundedness of $M$,
\begin{eqnarray*}
\left(\frac{1}{|R|}\int_RM_{1,T_{\O}}(f\chi_{{\mathbb R}^n\setminus 3R})^pdy\right)^{1/p}&\lesssim& \left(\frac{1}{|R|}\int_{2R}|T_{\O}(f\chi_{{\mathbb R}^n\setminus 3R})|^pdy\right)^{1/p}\\
&+&Mf(x)+M_{1,T_{\O}}f(x).
\end{eqnarray*}
Combining this estimate with
$$|T_{\O}(f\chi_{{\mathbb R}^n\setminus 3R})\chi_{2R}(y)|\lesssim Mf(x)+|T_{\O}(f\chi_{{\mathbb R}^n\setminus 6R})\chi_{2R}(y)|$$
and using also that, by H\"older's inequality, $M_{1,T_{\O}}f\le M_{p,T_{\O}}f$, we obtain
$$
\left(\frac{1}{|R|}\int_RM_{1,T_{\O}}(f\chi_{{\mathbb R}^n\setminus 3R})^pdy\right)^{1/p}\lesssim Mf(x)+M_{p,T_{\O}}f(x),
$$
which proves (\ref{point2}), and therefore, Theorem \ref{sharpbound} is completely proved.


\begin{thebibliography}{99}
\bibitem{B}
S.M. Buckley, {\it Estimates for operator norms on weighted spaces
and reverse Jensen inequalities}, Trans. Amer. Math. Soc., {\bf 340}
(1993), no. 1, 253--272.

\bibitem{CZ}
A.P. Calder\'on and A. Zygmund, {\it On singular integrals}, Amer. J. Math.  {\bf 78} (1956), 289--309.

\bibitem{CCPO}
J.M. Conde-Alonso, A. Culiuc, F. Di Plinio and Y. Ou, {\it A sparse domination principle for rough singular integrals},
Anal. PDE  {\bf 10} (2017), no. 5, 1255--1284.

\bibitem{DHL}
F. Di Plinio, T.P. Hyt\"onen and K. Li, {\it
Sparse bounds for maximal rough singular integrals via the Fourier transform}, preprint. Available at
https://arxiv.org/abs/1706.09064

\bibitem{G}
L. Grafakos, Classical Fourier analysis. Second edition. Graduate Texts in Mathematics, 249. {\it Springer, New York}, 2009.

\bibitem{HRT}
T.P. Hyt\"onen, L. Roncal and O. Tapiola, {\it Quantitative weighted estimates for rough homogeneous singular integrals},
Israel J. Math., {\bf 218} (2017), no. 1, 133--164.

\bibitem{L}
A.K. Lerner, {\it A weak type estimate for rough singular integrals}, preprint. Available at https://arxiv.org/abs/1705.07397

\bibitem{LPR}
T. Luque, C. P\'erez and E. Rela, {\it Optimal exponents in weighted estimates without examples}, Math. Res. Lett.  {\bf 22} (2015),  no. 1, 183--201.

\bibitem{S}
A. Seeger, {\it Singular integral operators with rough convolution kernels}, J. Amer. Math. Soc.  {\bf 9} (1996),  no. 1, 95--105.

\bibitem{St}
E.M. Stein, {\it Note on the class $L\log L$}, Studia Math. {\bf
32} (1969), 305--310.

\end{thebibliography}
\end{document}